\documentclass[12pt,bookmark=true]{article}

 \RequirePackage{geometry}
\usepackage{graphicx}				

\usepackage{amssymb}
\usepackage{amsmath}
\usepackage{amsthm}
\usepackage{enumerate}
\usepackage{color}
 \usepackage{lineno}

\newtheorem{theorem}{\indent Theorem}[section]

\newtheorem{lemma}[theorem]{\indent Lemma}

\makeatletter
\def\@fnsymbol#1{\ifcase#1\or 1\or 2\or 3\or 4\or 5\or 6\or 7\or 8\or 9\or 10\else\@ctrerr\fi}

\begin{document}
\title{Asymptotic behavior of the wave equation subject to a Kelvin-Voigt nonlocal damping}
\author{
J. C. O. Faria\footnote{Email address: jcofaria@uem.br. Corresponding author },\,\,\,\,M. M. Cavalcanti\footnote{Email address: mmcavalcanti@uem.br},\,\,\,\,V. N. Domingos Cavalcanti\footnote{Email address: vndcavalcanti@uem.br},\,\,\,\,C. A. Okawa\footnote{Email address: cintyaokawa@gmail.com}\\
	{\small Department of Mathematics, State University of
Maring\'a, Maring\'a, 87020-900,  P.  R, Brazil}
}

					
\maketitle

\begin{center}
\begin{abstract}
	In this article, we examine the well-posedness and asymptotic behavior of the energy associated with the wave equation that incorporates a Kelvin-Voigt nonlocal damping structure given by $-||\nabla u_t(t)||_2^2 \Delta u_t$. Utilizing the robust framework of nonlinear semigroups, we successfully demonstrate the existence of both strong and weak solutions. Our findings reveal that the decay rate for these solutions is optimally characterized by $1/t$, highlighting the effectiveness of this dissipative structure. This work not only enhances our understanding of the wave equation under nonlocal damping but also emphasizes the crucial balance between mathematical rigor and physical relevance.\\ 

\vspace{0.3cm}

\textbf{Keywords}: hyperbolic equations, asymptotic behavior, nonlocal damping term.

\textbf{Mathematics Subject Classification (2020)}: 35L05, 35L15, 35B35, 35B40.
\end{abstract}
\end{center}

\section{Introduction}

\setcounter{equation}{0}
\def\theequation{1.\arabic{equation}}\makeatother

The main purpose of the present article is to study the well-posedness and the asymptotic behavior of the energy associated with the following wave equation subject to a Kelvin-Voigt nonlocal damping:
\begin{equation}\label{eq1.1}
\left\{
\begin{aligned}
& \partial_t^2 u - \Delta u - D(u_t(t)) \Delta u_t =0 ~\hbox{ in }~\Omega \times \mathbb{R}_+,\\
&u=0~ \hbox{ on }~\partial \Omega \times \mathbb{R}_+,\\
&u(x,0)=u_0(x);\quad \partial_t u(x,0)=u_1(x),~x\in \Omega,
\end{aligned}
\right.
\end{equation}
where
\begin{align}\label{D(u(t))}
D(u_t(t)) = \|\nabla u_t(t)\|_{L^2(\Omega)}^2
\end{align}
and $\Omega$ is a bounded domain in $\mathbb{R}^n$ with a smooth boundary.

The study of nonlocal dissipative phenomena dates back to 1980s, marking a significant advancement in this field. In 1989, Balakrishnan and Taylor \cite{Balakrishnan}, inspired by Zhang's work \cite{Zhang}, introduced several models related to the investigation of these dissipative phenomena, specifically in the context of flight structures. These models are known as nonlocal dissipations. The one-dimensional behavior of such models can be described by the following ODE:
$$x''(t) + \alpha x(t) + \beta D(x(t),x'(t))=0,$$
where $x(t)$ represents the displacement of a point $x$ in the flight structure at time $t$ and $\alpha, \beta$ are positive constants. This model has attracted the attention of researchers worldwide, particularly when the structural damping term is given by by $D(x(t),x'(t))x'(t)$, due to the slow decay produced by this type of damping mechanism. Indeed, let $\Omega$ be  a bounded domain in $\mathbb{R}^n$ with smooth boundary and let us consider, for instance, the linear wave equation subject to a nonlinear and nonlocal damping term of Balakrishnan and Taylor's type:
\begin{eqnarray}\label{eq1.1'}
	\partial_t^2 u- \Delta u + E_u^w(t) \partial_t u = 0 ~\hbox{ in }~\Omega \times \mathbb{R}_+,
\end{eqnarray}
with Dirichlet homogeneous boundary conditions and weak initial data $\{u_0,u_1\}\in H_0^1(\Omega) \times L^2(\Omega)$. The energy associated to weak solutions to problem (\ref{eq1.1'}) is given by $E_u^w(t):= \frac12 ||\partial_t u(t)||_2^2 + \frac12||\nabla u(t)||_2^2$.

A straightforward computation shows that
\begin{align*}
\frac{d}{dt} E^w_u(t) + E_u^w(t) \|\partial_t u(t)\|_{L^2(\Omega)}^2=0
\end{align*}
for all $t\geq 0$ and
\begin{align}\label{eq1.2}
\frac{d}{dt} E_u^w(t) + 2\left(E^w_u(t)\right)^2 \geq 0
\end{align}
for any $t\geq 0.$ Thus, we infer from inequality \eqref{eq1.2} that
\begin{align}\label{eq1.3}
E^w_u(t) \geq \left( \frac{1}{E^w_u(0)} + 2t\right)^{-1}
\end{align}
for any $t\geq 0,$ which means that the energy $E^w_u(t)$ can not decay fast to zero. In fact, it can be proven that the energy associated with problem (\ref{eq1.1'}) satisfies the following inequality:
\begin{align*}
\left(4t+\frac{1}{E^w_u(0)}\right)^{-1}\leq E_u^w(t)\leq \left(\frac{1}{\mu}(t-1)^{+}+E_u^w(0)^{-1}\right)^{-1}
\end{align*}
for some $\mu>0$ and for any $t\geq 0$, which means that $1/t$ is the best decay rate estimate for this kind of prototype.

On the other hand, the weak and regular energies associated to the equation (\ref{eq1.1}) are expressed as follows:
\begin{eqnarray}
	E_u^w(t)&:=& \frac12 || u_t(t)||_{L^2(\Omega)}^2 + \frac12||\nabla u(t)||_{L^2(\Omega)}^2 \label{WE}\\
	E_u^r(t)&:=& \frac12 ||\nabla u_t(t)||_{L^2(\Omega)}^2 + \frac12||\Delta u(t)||_{L^2(\Omega)}^2.\label{RE}
\end{eqnarray}
From these definitions, we can derive the following energy identities:
\begin{eqnarray}
	E_u^w(T)- E_u^w(s) = - \int_S^T D(u_t(t)) ||\nabla  u_t(t)||_{L^2(\Omega)}^2\,dt,~0<S<T<+\infty,\label{WIE}\\
	E_u^r(T)- E_u^r(s) = - \int_S^T D(u_t(t)) ||\Delta u_t(t)||_{L^2(\Omega)}^2\,dt,~0<S<T<+\infty,\label{RIE}
\end{eqnarray}
which establish that both energies are weakly and regularly dissipative, as well as suggest that the previous analysis could also be applicable in this situation.

There are some interesting works in the previous literature addressing various analytical aspects of a wide range of evolution equations subject to Balakrishnan and Taylor dissipations. We would like to mention the following $n$-dimensional model for beam equations with nonlocal Kelvin-Voigt damping:
$$\partial_t^2 u + \Delta^2 u - M(||\nabla u(t)||_2^2)\Delta u - N(||\nabla u(t)||_2^2)\Delta u_t=0.$$

In their work, da Silva and Narciso \cite{Silva-Narciso} shown that the nonlocal damping function $N(s)$ is bounded below by
a positive constant $\alpha_0 > 0$ for all $s \geq 0$, which implies the damping does not degenerate over time, resulting in exponential decay. {However, there are other significant studies in which structural damping can degenerate.  In \cite{Li-Zhang-Hu}, Li, Zang and Hu explored the following nonlinear wave equation, which incorporates both nonlocal damping and nonlinear boundary damping, posed in a bounded domain $\Omega\subset\mathbb{R}^n$ and is based on the nonlinear theory of meson fields:
\begin{equation*}
	\left\{
	\begin{aligned}
		& u_{tt} - \Delta u + M\left(\int_{\Omega}|u_t|^2dx\right)u_t+N\left(\int_{\Omega}|u|^pdx\right)|u|^{p-2}u =0 ~\hbox{ in }~\Omega \times \mathbb{R}_+,\\
		&u=0~ \hbox{ on }~\Gamma_0 \times \mathbb{R}_+,\\
			&\frac{\partial u}{\partial\nu}+\alpha(t)h(u_t)~ \hbox{ on }~\Gamma_1 \times \mathbb{R}_+,\\
		&u(x,0)=u_0(x);\quad u_t(x,0)=u_1(x),~x\in \Omega.
	\end{aligned}
	\right.
\end{equation*}
In this formulation, $\partial \Omega = \Gamma_0 \cup \Gamma_1$, $ \overline{\Gamma}_0 \cap \overline{\Gamma}_1 = \emptyset $,   $M(s)$ is a nonlocal positive non-decreasing function that is bounded below by a positive constant $M_0>0$. Additionally $N$, $\alpha$ and $h$ are real
valued functions that satisfy certain general conditions. The authors used the multiplier technique to derive a general result regarding energy decay associated with the solution. They also established energy decay for the equation that includes nonlocal damping alone, but without the nonlinear boundary damping. Further works related to this topic can be found in \cite{Liu-Peng-Su}, \cite{Hu-Li-Liu}, \cite{Zhang-Li} and references therein. }
  
  As far as we are concerned, the present paper is pioneering in addressing a nonlinear and nonlocal damping structure defined as  $-||\nabla u_t(t)||_2^2 \Delta u_t$ which differs from the other ones $E_u^w(t)u_t$, $-E_u^w(t) \Delta u_t$,  $||\nabla u(t)||_2^2 u_t$, $-||\nabla u(t)||_2^2 \Delta u_t $, $||\nabla u(t)||_2^2 ||u_t||_2^{m-2}u_t$, $-[E_u^w(t)]^q \Delta u_t$. It is worth mentioning that our nonlocal term $||\nabla u_t(t)||_2^2 $ incorporates the potential part of the strong energy, in contrast to the aforementioned models, where the nonlocal component of the damping structure accounts for some aspects of the weak energy, or in some cases, the total weak energy of the system.

 An interesting and natural problem arises when we have a dissipative term of the type $a(x) E_u^w(t) u_t$ or $-E_u^w(t)\hbox{div}[a(x) \nabla u_t]$. In these cases, the functions that control the dissipation may degenerate within the domain $\Omega$ in the spacial variable $x$ or at infinity in the temporal variable $t$. Specifically, we have $a(x) \geq a_0 >0$ in a neighbourhood $\omega$ of the boundary $\partial \Omega$. What can we expect in this case?  Unfortunately, single dissipations locally distributed like the above examples remain an open problem. In general, we employ microlocal analysis in order to propagate, roughly speaking, the microlocal defect measure of a certain sequence $\{\partial_t y^k\}_k$ which converges weakly to zero in $L^2(\Omega \times (0,T))$ (see \cite{Math Annalen}). However, without information on where $\partial_t y^k$ converges strongly to zero in a subset $\omega \times (0,T)$ such that $\omega\times (0,T) \subset \Omega \times (0,T)$, in view of the degeneracy of $a(x) E_u^w(t) u_t$, it is impossible to use arguments of microlocal analysis for propagating this convergence to the whole space $\Omega \times (0,T).$ In other words, $a(x) E_u^w(t)$ is not bounded from below for sufficiently large values of  $t$, which destroys any attempt to use microlocal analysis. In addition, the use of suitable multipliers as in \cite{TAMS},\cite{ARMA} seems do not work. We are not aware whether the strategies employed in \cite{Ammari-Robbiano}, \cite{Bellassoued} or \cite{Burq} can be applied in the present case to obtain a very slow decay of logarithmic type due to the nonlinear nature of the equation $$u_{tt} - \Delta u + a(x) E_u^w(t) u_t=0~\hbox{ in } \Omega \times \mathbb{R}_+.$$

It is also important to highlight that previous literature has often relied on the theory of linear semigroups or a Galerkin scheme to establish the well-posedness of related problems. In instances where the semigroup method was not applicable, a Galerkin scheme sometimes only provided the existence and uniqueness of regular solutions, making it impossible to prove the existence of weak solutions. In this paper, we need to address a nonlinear operator, which presents unique challenges. Although both the nonlocal term and the dissipative term are inherent to regular solutions, a thorough analysis of certain a priori estimates, along with identities \eqref{WE} - \eqref{RIE}, reveals a significant limitation: it is impossible to achieve the necessary level of regularity to utilize essential compactness results when applying a Galerkin scheme to demonstrate the existence of solutions to the problem at hand. Nevertheless, we successfully introduce an approach that merges the theory of nonlinear semigroups combined with a Galerkin analysis to establish the existence of both strong and weak solutions, as well as a decay rate of the form $1/t$ for both regular and weak solutions, considered the best for this type of dissipative structure. 

This work is organized as follows: section 2 is devoted to show the well-posedness of problem (\ref{eq1.1}) by considering the nonlinear semigroup approach, while in section 3 we prove the regular and weak solutions decay as $1/t$ as $t$ tends to infinity.   Finally in section 4 we wrote an appendix concerning the well-known nonlinear semigroup theory for $m$-dissipative nonlinear operators in Banach spaces necessary for the development of this article.

\medskip
\section{Wellposedness}
\medskip
In order to address the wellposedness of problem \eqref{eq1.1}, consider
\begin{equation*}
U:=\left(
\begin{aligned}
&u\\
&u_{t}
\end{aligned}
\right).
\end{equation*}
Therefore, \eqref{eq1.1} can be rewritten as
\begin{equation}\label{eq2.1}
\frac{dU}{dt}=
\left(
\begin{aligned}
&0\qquad \quad I\\
&\Delta \quad ||\nabla \cdot||_2\Delta
\end{aligned}
\right)
\left(
\begin{aligned}
&u\\
&u_t
\end{aligned}
\right).
\end{equation}

Motivated by equation (\ref{eq2.1}), we define the operator $A$ on the space $\mathcal{H}:= H_0^1(\Omega) \times L^2(\Omega)$, where
\begin{eqnarray}\label{eq2.2}
D(A)=\{(u,v)\in \mathcal{H}: v\in H_0^1(\Omega) ~\hbox{and}~(\Delta u+||\nabla v||_2^2 \Delta v)\in L^2(\Omega)\},
\end{eqnarray}
and
\begin{eqnarray}\label{eq2.3}
A(u,v)= \left(v, \Delta u+||\nabla v||_2^2\Delta v\right)~\forall (u,v)\in D(A).
\end{eqnarray}

In the following discussion, we aim to utilize standard results from nonlinear semigroup theory. Indeed, to prove that $A$ is a \underline{dissipative} operator, let us consider $(u_1,v_1), (u_2,v_2) \in D(A)$. Then, after performing integration by parts, we have:
\begin{eqnarray}\label{eq2.4}
&&\left<A(u_1,v_1)-A(u_2,v_2) , (u_1,v_1)-(u_2,v_2)\right>_{\mathcal{H}}\\
&&=-\left\{||\nabla v_1||_2^4 - ||\nabla v_1||_2^2 (\nabla v_1,\nabla v_2)_{L^2(\Omega)} - ||\nabla v_2||_2^2 (\nabla v_2,\nabla v_1)_{L^2(\Omega)} +||\nabla v_2||_2^4\right\}.\nonumber
\end{eqnarray}

On the other hand, we observe that
\begin{eqnarray}\label{eq2.5}
&&||\nabla v_1||_2^4 - ||\nabla v_1||_2^2 (\nabla v_1,\nabla v_2)_{L^2(\Omega)} - ||\nabla v_2||_2^2 (\nabla v_2,\nabla v_1)_{L^2(\Omega)} +||\nabla v_2||_2^4\\
&&\geq  ||\nabla v_1||_2^4 - \frac12 ||\nabla v_1||_2^4 - \frac12 ||\nabla v_1||_2^2||\nabla v_2||_2^2 - \frac12 ||\nabla v_2||_2^4 - \frac12 ||\nabla v_2||_2^2||\nabla v_1||_2^2 + ||\nabla v_2||_2^4\nonumber\\
&&=\frac12\left[||\nabla v_1||_2^2-||\nabla v_2||_2^2\right]^2\geq0.\nonumber
\end{eqnarray}

Therefore, combining (\ref{eq2.4}) and (\ref{eq2.5}) we deduce that
\begin{eqnarray}\label{eq2.6}
\left<A(u_1,v_1)-A(u_2,v_2) , (u_1,v_1)-(u_2,v_2)\right>_{\mathcal{H}}\leq 0, ~\hbox{ for all }(u_1,v_1), (u_2,v_2) \in D(A),
\end{eqnarray}
which proves the desired dissipativity.

Next, we will demonstrate that the operator $A$ is \underline{$m$-dissipative}, or equivalently, that:
\begin{eqnarray}\label{eq2.7}
R[I-\alpha A]= \mathcal{H},~\forall \alpha>0.
\end{eqnarray}

Indeed, let us consider $(f,g)\in \mathcal{H}=H_0^1(\Omega) \times L^2(\Omega)$. We aim to show that there exists $(u,v)\in D(A)$ such that
\begin{eqnarray}\label{eq2.8}
(u,v)-\alpha \left(v,~\Delta u+||\nabla v||_2^2 \Delta v\right)= (f,g)
\end{eqnarray}

Taking (\ref{eq2.8}) into account, we can derive the following system:
\begin{equation}\label{eq.2.9}
\left\{
\begin{aligned}
& u-\alpha v=f\\
&v-\alpha \Delta u- \alpha ||\nabla v||_2^2 \Delta v=g.
\end{aligned}
\right.
\end{equation}

From the first equation of (\ref{eq.2.9}) $u$ can be expressed as follows:
\begin{eqnarray}\label{eq2.10}
&&u=f+\alpha v.
\end{eqnarray}

By substituting equation (\ref{eq2.10}) into the second equation of (\ref{eq.2.9}), we arrive at the following nonlinear elliptic problem:
\begin{eqnarray}\label{eq2.11}
-\alpha^2 \Delta v + v - \alpha ||\nabla v||_2^2 \Delta v =g+\alpha \Delta f.
\end{eqnarray}

The following calculations will focus on finding a solution to problem (\ref{eq2.11}) in order to establish the $m$-dissipative property of the operator in question. For this, we shall make use of a Galerkin scheme. Let $\{\omega_j\}_{j\in \mathbb{N}}$ be a special basis in $H_0^1(\Omega) \cap H^2(\Omega)$ formed by the eigenfunctions of $-\Delta$, namely, $-\Delta \omega_j=\lambda_j \omega_j$, $j=1,2,\cdots$ and let $V_m=\left[\omega_1,\cdots,\omega_m\right]=\hbox{span}\{\omega_1,\cdots,\omega_m\}$. In $V_m$ we want to find $$v_m=\sum_{i=1}^m \xi_i \omega_i\in V_m,$$ such that
\begin{eqnarray}\label{eq2.12}
-\alpha^2\left(\Delta v_m, \omega_j\right)_{L^2(\Omega)} &+& \left(v_m, \omega_j\right)_{L^2(\Omega)} - \alpha ||\nabla v_m||_2^2\left( \Delta v_m, \omega_j\right)_{L^2(\Omega)}\\
&=& \left(g, \omega_j\right)_{L^2(\Omega)} + \alpha \left< \Delta f, \omega_j\right>_{H^{-1};H_0^1},~j=1,2,\cdots,m.\nonumber
\end{eqnarray}

Next, our objective is to prove that system (\ref{eq2.12}) is well-posed in $\mathbb{R}^m$. For this purpose we are going to make use of the following helpful result:
\begin{lemma}\label{Lema1}[Visik]~Let $\xi \mapsto P(\xi)$ be a continuous map $\mathbb{R}^n \rightarrow \mathbb{R}^n$ such that for some $\rho>0$ we have $\left(P(\xi),\xi\right)_{\mathbb{R}^n}\geq 0$, for all $\xi \in \mathbb{R}^n$ with $||\xi||=\rho$. Then, there exists $\xi_0 \in \overline{B_\rho(0)}$ such that $P(\xi_0)=0.$
\end{lemma}
\begin{proof}
See J.L. Lions \cite{Lions}, Lemme 4.3 on pg. 53.
\end{proof}

Indeed, once  $\{\omega_j\}_{j\in \mathbb{N}}$ is a special basis, one has
\begin{equation}\label{eq2.13}
\left.
\begin{aligned}
&\quad\left(\omega_j\right) ~\hbox{ is a complete orthonormal system in }~ L^2(\Omega),\\
&\,\left(\frac{\omega_j}{\sqrt{\lambda_j}}\right)~\hbox{ is a  complete orthonormal system in } H_0^1(\Omega),\\
&\quad\left(\frac{\omega_j}{\lambda_j}\right)~\hbox{ is a complete orthonormal system in }H_0^1(\Omega) \cap H^2(\Omega),
\end{aligned}
\right.
\end{equation}
from which we deduce that
\begin{eqnarray}\label{eq2.14}
&&||v_m||_{H_0^1(\Omega)}^2= ((v_m, v_m))_{H_0^1(\Omega)}= \left(\left(\sum_{j=1}^m \xi_j \omega_j,\sum_{i=1}^m \xi_i \omega_i \right)\right)_{H_0^1(\Omega)}= \sum_{i=1}^m \xi_i^2 \lambda_i\\
&&||v_m||_{L^2(\Omega)}^2= (v_m, v_m)_{L^2(\Omega)}= \left(\sum_{j=1}^m \xi_j \omega_j,\sum_{i=1}^m \xi_i \omega_i\right)_{L^2(\Omega)}= \sum_{i=1}^m \xi_i^2.\label{eq2.15}
\end{eqnarray}

Substituting $v_m=\sum_{i=1}^m \xi_i \omega_i$ in (\ref{eq2.12}), observing that $-\Delta \omega_i=\lambda_i \omega_i$ and taking (\ref{eq2.13}) into account, we deduce
\begin{eqnarray}\label{eq2.16}
&&\alpha^2 \xi_i \lambda_i + \xi_i + \alpha\left(\sum_{i=1}^m \xi_i^2 \lambda_i\right) \xi_i \lambda_i\\
&&=\left(g, \omega_i\right)_{L^2(\Omega)} + \alpha \left< \Delta f, \omega_i\right>_{H^{-1};H_0^1},~i=1,2,\cdots,m.\nonumber
\end{eqnarray}

Now, let us define, for each $i=1,\cdots,m$:
\begin{eqnarray*}
&&\eta_i=\alpha^2 \xi_i \lambda_i + \xi_i + \alpha\left(\sum_{i=1}^m \xi_i^2 \lambda_i\right) \xi_i \lambda_i
-\left(g, \omega_i\right)_{L^2(\Omega)} - \alpha \left< \Delta f, \omega_i\right>_{H^{-1};H_0^1}
\end{eqnarray*}
and consider the map
\begin{eqnarray}
P&:&\mathbb{R}^m \rightarrow \mathbb{R}^m\\
\xi&=&(\xi_1, \cdots, \xi_m) \mapsto P(\xi) = \eta= (\eta_1, \cdots, \eta_m),
\end{eqnarray}
which is clearly continuous. To apply Lemma \ref{Lema1}, we
 aim to establish the existence of $\rho_0>0$ such that $(P\xi,\xi)_{\mathbb{R}^m}\geq 0$ for all $\xi\in \mathbb{R}^m$ with $||\xi||=\rho_0.$~In fact, we have:
\begin{eqnarray*}
(P\xi,\xi)&=&(\eta,\xi) \\
&=& \sum_{i=1}^m\left( \alpha^2 \xi_i^2 \lambda_i + \xi_i^2 + \alpha\left(\sum_{i=1}^m \xi_i^2 \lambda_i\right) \xi_i^2 \lambda_i\right)\\
&-& \sum_{i=1}^m (g, \xi_i\omega_i )_{L^2(\Omega)}- \alpha \left< \Delta f, \xi_i\omega_i \right>_{H^{-1};H_0^1}.
\end{eqnarray*}
From this, it follows that
\begin{eqnarray}\label{eq2.17}
(P\xi,\xi)&=& \alpha^2 ||v_m||_{H_0^1(\Omega)}^2 + ||v_m||_{L^2(\Omega)}^2 + \alpha ||v_m||_{H_0^1(\Omega)}^4-(g,v_m)_{L^2(\Omega)}-\alpha\left<\Delta f, v_m \right>_{H^{-1};H_0^1}\\
&\geq&  \alpha^2 ||v_m||_{H_0^1(\Omega)}^2 + ||v_m||_{L^2(\Omega)}^2 - ||g||_2 ||v_m||_2 - ||\Delta f||_{H^{-1}}||v_m||_{H_0^1}.\nonumber
\end{eqnarray}

We need to consider two cases:

(i) ~If $v_m=0$, then it implies that $(P\xi,\xi)\geq0$ for any $\rho>0$.

(ii)~If $v_m\ne 0$ then we have
\begin{eqnarray*}
&& \alpha^2 ||v_m||_{H_0^1(\Omega)}^2  - ||\Delta f||_{H^{-1}}||v_m||_{H_0^1} \geq 0 ~\hbox{ if }~ ||v_m||_{H_0^1(\Omega)} \geq \frac{||\Delta f||_{H^{-1}}}{\alpha^2}=c_1,\\
&& ||v_m||_{L^2(\Omega)}^2 - ||g||_2 ||v_m||_2 \geq 0 ~\hbox{ if }~||v_m||_{L^2(\Omega)} \geq ||g||_2=c_2.
\end{eqnarray*}

Let us define $\beta_m :=\hbox{min}\{\lambda_1, \cdots,\lambda_m\}$. Then, from (\ref{eq2.14}) and (\ref{eq2.15}) we obtain
\begin{eqnarray}\label{eq2.18}
||v_m||_{H_0^1(\Omega)}^2 &=&\sum_{i=1}^m \xi_i^2 \lambda_i \geq \beta_m \sum_{i=1}^m \xi_i^2= \beta_m ||\xi||_{\mathbb{R}^m}^2,\\
||v_m||_{L^2(\Omega)}^2 &=&\sum_{i=1}^m \xi_i^2 = ||\xi||_{\mathbb{R}^m}^2.\label{eq2.19}
\end{eqnarray}

We affirm that there exist two positive constants $c_0$ and  $\rho_0$ such that for $\xi \in \mathbb{R}^m$ with $||\xi||_{\mathbb{R}^m}=\rho_0$, the following inequality holds:
\begin{eqnarray}\label{eq2.20}
||v_m||_{H_0^1(\Omega)}^2 + ||v_m||_{L^2(\Omega)}^2 > c_0.
\end{eqnarray}

Indeed, if $\rho_1>0$ is such that $\rho_1 > \frac{c_1}{\sqrt{\beta_m}}$ then, for all $\xi\in \mathbb{R}^m$ with $||\xi||_{\mathbb{R}^m}=\rho_1$, it follows from (\ref{eq2.18}) that $$||v_m||_{H_0^1(\Omega)}^2 \geq \beta_m \rho_1^2 > \beta_m \frac{c_1^2}{\beta_m}=c_1^2.$$

Similarly, if $\rho_2>0$ is such that $\rho_2>c_2$, then, for all $\xi\in \mathbb{R}^m$ with $||\xi||_{\mathbb{R}^m}=\rho_2$, we conclude from (\ref{eq2.19}) that $$||v_m||_{L^2(\Omega)}^2 = ||\xi||_{\mathbb{R}^m}^2 = \rho_2^2 > c_2^2.$$

Now, let $\rho_0 > 0$ be chosen  such that $\rho_0 > \hbox{max} \{\rho_1, \rho_2\}$. Then for all $||\xi||_{\mathbb{R}^m}=\rho_0$ we have $$||v_m||_{H_0^1(\Omega)}^2 + ||v_m||_{L^2(\Omega)}^2 > c_1^2 + c_2^2 := c_0>0,$$ which proves the statement in (\ref{eq2.20}). As a consequence, by applying Lemma \ref{Lema1} there exists $\xi_0 \in \overline{B_\rho(0)}$ such that $P(\xi_0)=0$ which means that the  system (\ref{eq2.12}) is well-posed.

To complete this step of the proof, we will now pass to the limit in the approximated problem (\ref{eq2.12}).

\medskip
\noindent{\bf A priori estimates}
\medskip

Multiplying (\ref{eq2.12}) by $\xi_j$ and summing from $1$ to $m$, we deduce
\begin{eqnarray}\label{eq2.21}
\alpha^2||\nabla v_m||_2^2 + ||v_m||_2^2 + \alpha||\nabla v_m||_2^4 &=& -\alpha (\nabla f, \nabla v_m)_{L^2(\Omega)} + (g, v_m)_{L^2(\Omega)}\\
&\leq& \frac{\alpha^2}{4\eta} ||\nabla f||_2^2 + \eta ||\nabla v_m||_2^2 + \frac{1}{2} ||g||_2^2 + \frac12 ||v_m||^2.\nonumber
\end{eqnarray}

Choosing $\eta=\alpha/2$, we obtain
\begin{eqnarray}\label{eq2.22}
\frac{\alpha^2}2||\nabla v_m||_2^2 + \frac12 ||v_m||_2^2 + \alpha||\nabla v_m||_2^4
&\leq& 2\alpha ||\nabla f||_2^2  + \frac{1}{2} ||g||_2^2.
\end{eqnarray}

From (\ref{eq2.22}) and we deduce that:
\begin{eqnarray}\label{eq2.23}
\{v_m\}~\hbox{ is bounded in  }H_0^1(\Omega),
\end{eqnarray}
and since the embedding $H_0^1(\Omega) \hookrightarrow L^2(\Omega)$ is compact,
we conclude from (\ref{eq2.23}) that there exists subsequence of $\{v_m\}$, which from now on will be represent by the same notation, such that:
\begin{eqnarray}\label{eq2.24}
&&v_m \rightarrow v ~\hbox{ strongly in } L^2(\Omega), ~\hbox{ as well as}\\
&&v_m \rightharpoonup v ~\hbox{ weakly in }~H_0^1(\Omega),\label{eq2.25}\\
&& ||\nabla v_m||_2^2 \rightarrow \chi ~\hbox{ in }~\mathbb{R}.\label{eq2.26}
\end{eqnarray}

The convergences in (\ref{eq2.24}), (\ref{eq2.25}) and (\ref{eq2.26}) allow us pass to the limit in the approximated problem (\ref{eq2.12}) in order to obtain:
\begin{eqnarray}\label{eq2.27}
-\alpha^2 \Delta v + v - \alpha \chi \Delta v = \alpha \Delta f + g ~\hbox{ in }~H^{-1}(\Omega).
\end{eqnarray}

Next, our goal is to prove that:
\begin{eqnarray}\label{eq2.28}
&&\chi=||\nabla v||_2^2.
\end{eqnarray}

Indeed, from the approximated problem (\ref{eq2.12}) we can write
\begin{eqnarray}\label{eq2.29}
&&\alpha^2||\nabla v_m||_2^2 + ||v_m||_2^2 + \alpha||\nabla v_m||_2^4 = -\alpha (\nabla f, \nabla v_m)_{L^2(\Omega)} + (g, v_m)_{L^2(\Omega)}.
\end{eqnarray}
Thus, by considering the convergences in (\ref{eq2.24}), (\ref{eq2.25}), and (\ref{eq2.26}), along with the identity in (\ref{eq2.27}), we can deduce:
\begin{eqnarray}\label{eq2.30}
\lim_{m\rightarrow +\infty} \left(\alpha^2||\nabla v_m||_2^2  + \alpha||\nabla v_m||_2^4\right) +  ||v||_2^2 &=&  -\alpha (\nabla f, \nabla v)_{L^2(\Omega)} + (g, v)_{L^2(\Omega)}\\
&=& \left< -\alpha^2 \Delta v + v + \alpha \chi \Delta v, v\right>_{H^{-1};H_0^1}.\nonumber\\
&=& \alpha^2 ||\nabla v||_2^2 + ||v||_2^2 + \alpha \chi||\nabla v||_2^2.\nonumber
\end{eqnarray}

On the other hand, since $||\nabla v_m||_2^2 \rightarrow \chi $ as $m\rightarrow + \infty$, we have
\begin{eqnarray}\label{eq2.31}
 \alpha^2||\nabla v_m||_2^2+\alpha||\nabla v_m||_2^4\rightarrow \alpha^2\chi+\alpha \chi^2~\hbox{ as }m\rightarrow + \infty.
\end{eqnarray}

By combining (\ref{eq2.30}) and (\ref{eq2.31}) and from uniqueness of the limit we deduce that
$$\alpha^2 ||\nabla v||_2^2 +  \alpha \chi||\nabla v||_2^2 = \alpha^2\chi+\alpha \chi^2 \Leftrightarrow \left(\alpha^2+\alpha\chi\right)||\nabla v||_2^2 = \left(\alpha^2+\alpha\chi\right)\chi,$$ which proves the statement in (\ref{eq2.28}). Consequently from (\ref{eq2.27}) we deduce
\begin{eqnarray}\label{eq2.32}
-\alpha^2 \Delta v + v - \alpha ||\nabla v||_2^2 \Delta v = \alpha \Delta f + g ~\hbox{ in }~H^{-1}(\Omega).
\end{eqnarray}

Returning to (\ref{eq2.10}) we have that $u=f+\alpha v$, where $v$ is solution to problem (\ref{eq2.32}). Remember that we need to prove that $(u,v) \in D(A)$, that is, we must show that  $\Delta u + ||\nabla v||_2^2 \Delta v \in L^2(\Omega)$. In other words, we need to prove that
\begin{eqnarray}\label{eq2.33}
\Delta f + \alpha \Delta v + ||\nabla v||_2^2 \Delta v \in L^2(\Omega).
\end{eqnarray}

Notice that this statement holds true because, according to equation (\ref{eq2.32}), we have  that $ \Delta f+\alpha \Delta v +  ||\nabla v||_2^2 \Delta v = \frac{1}{\alpha}(v-g)\in L^2(\Omega)$, which proves that the operator $A$ is $m$-dissipative, as we intended to establish in equation (\ref{eq2.7}). Consequently, according to the theory of nonlinear semigroups (see theorem \ref{main theo} in appendix) there exists a nonlinear semigroup $\{T(t):t\geq0\}$ on $\overline{D(A)}^{\mathcal{H}}=H_0^1(\Omega)\times L^2(\Omega)$.

In fact, based on theorem \ref{TeoApp3} in the appendix and the fact that $A$ is $m$-dissipative, the problem:
\begin{equation}
\left\{
\begin{aligned}
& \frac{dU}{dt}= A_\alpha U,~t\geq0\\
&U(0)= U_0=(u_0,u_1)\in \mathcal{H},
\end{aligned}
\right.
\end{equation}
where $A_\alpha$ is the Yosida operator associated to the operator $A$, possesses a unique solution $U\in C^1([0,+\infty); \mathcal{H})$ for $U_0\in \mathcal{H}$. We also observe that from (\ref{3.6}) in the proof of the Theorem \ref{main theo} in the appendix, we have
\begin{eqnarray}\label{3.6}
T(t)U_0:= \lim_{\alpha \rightarrow 0} T_\alpha(t)U_0,~\hbox{ for }~U_0\in \overline{D(A)}= \mathcal{H},
\end{eqnarray}
which implies that the nonlinear semigroup $T(t)$  is obtained, by continuity, through the extension from $D(A)$ to $\overline{D(A)}= \mathcal{H}$.
As a consequence, problem (\ref{eq1.1}) has a unique mild solution $U(t) = S(t)U_0$, $t\geq0$ where $U=(u_0,u_1)\in \mathcal{H}$ and this solution is regular if $U_0\in D(A)$ (see theorem \ref{TeoApp1} (iv)).\\

Based on this previous discussion, we can now state the first main result of this paper:
\begin{theorem}\label{theorem1}
If $U_0:=(u_0,u_1)\in D(A)$ defined in (\ref{eq2.2})-(\ref{eq2.3}), problem (\ref{eq1.1}) has a unique regular solution $U=(u,u_t)$ in the class $$U\in C^1([0,T]; D(A)).$$

In addition, if $U_0:=(u_0,u_1)\in \mathcal{H}$, problem (\ref{eq1.1}) has a unique mild solution $U=(u,u_t)$ in the class $$U\in C^1([0,T]; \mathcal{H}).$$
\end{theorem}

\medskip
\section{Asymptotic behavior}
\medskip

From now on, our goal is to prove the asymptotic behavior of the problem at hand. Multiplying the first equation in (\ref{eq1.1}) by $E_u^w(t) u$ and performing integrations by parts, we obtain
\begin{eqnarray}\label{eq3.1}
&&2\int_S^T \left[E_u^w(t)\right]^2\,dt = 2\int_S^T E_u^w(t)||u_t(t)||_2^2 \,dt + \int_S^T \left[E_u^w(t)\right]' \left(u_t(t),u(t)\right)_{L^2(\Omega)}\,dt\\
&&-\int_S^T E_u^w(t)||\nabla u_t(t)||_2^2 \left(\nabla u_t(t),\nabla u(t)\right)_{L^2(\Omega)}\,dt - \left[E_u^w(t)\left(u_t,u(t)\right)_{L^2(\Omega)} \right]_S^T,\nonumber
\end{eqnarray}
for $0\leq S \leq T < +\infty$.

In addition, from the energy identity (\ref{WIE})  we have
\begin{eqnarray}\label{eq3.2}
E_u^w(T) + \int_S^T ||\nabla u_t(t)||_2^4 \,dt = E_u^w(S).
\end{eqnarray}

To achieve our aim, we need to prove that
\begin{eqnarray}\label{eq3.3}
2\int_S^T E_u^w(t)||u_t(t)||_2^2 \,dt \lesssim \int_S^T ||\nabla u_t(t)||_2^4 \,dt + \hbox{''favorable terms"},
\end{eqnarray}
since the remaining terms in (\ref{eq3.1}) are suitable for advancing our proof strategy.

Indeed, we first observe that
\begin{eqnarray}\label{eq3.4}
2\int_S^T E_u^w(t)||u_t(t)||_2^2 \,dt &=& \int_S^T \left(||u_t||_2^2+||\nabla u(t)||_2^2\right)||u_t(t)||_2^2\,dt\\
&=& \int_S^T ||u_t(t)||_2^4 \,dt + \int_S^T ||\nabla u(t)||_2^2 ||u_(t)||_2^2\,dt\nonumber\\
&\leq& \lambda_1^2 \int_S^T ||\nabla u_t(t)||_2^4\,dt + \int_S^T ||\nabla u(t)||_2^2 ||u_t(t)||_2^2\,dt,\nonumber
\end{eqnarray}
where $\lambda_1>0$ comes from Poincar\'e's inequality.

On the other hand, we also have
\begin{eqnarray}\label{eq3.5}
 \int_S^T ||\nabla u(t)||_2^2 ||u_t(t)||_2^2\,dt &\leq& 2\int_S^T E_u^w(t) ||u_t(t)||_2^2\,dt\\
 &\leq& 2\lambda_1\int_S^T E_u^w(t) ||\nabla u_t(t)||_2^2\,dt\nonumber\\
 &\leq& \lambda_1^2\int_S^T ||\nabla u_t(t)||_2^4\, dt + \int_S^T [E_u^w(t)]^2\,dt.\nonumber
\end{eqnarray}

Combining (\ref{eq3.4}) and (\ref{eq3.5}) we deduce that
\begin{eqnarray}\label{eq3.6}
2\int_S^T E_u^w(t)||u_t(t)||_2^2 \,dt\leq 2\lambda_1^2 \int_S^T ||\nabla u_t(t)||_2^4\, dt + \int_S^T \left[E_u^w(t)\right]^2\,dt,
\end{eqnarray}
which states (\ref{eq3.3}). 

Thus, from (\ref{eq3.1}) and (\ref{eq3.6}) we can conclude:
\begin{eqnarray}\label{eq3.7}
&&\int_S^T \left[E_u^w(t)\right]^2\,dt \leq \int_S^T \left[E_u^w(t)\right]' \left(u_t(t),u(t)\right)_{L^2(\Omega)}\,dt\\
&&-\int_S^T E_u^w(t)||\nabla u_t(t)||_2^2 \left(\nabla u_t(t),\nabla u(t)\right)_{L^2(\Omega)}\,dt - \left[E_u^w(t)\left(u_t,u(t)\right)_{L^2(\Omega)} \right]_S^T\nonumber\\
&&+ 2\lambda_1^2 \int_S^T ||\nabla u_t(t)||_2^4\, dt .\nonumber
\end{eqnarray}

In the subsequent calculations, we shall estimate the terms in the right side of inequality \eqref{eq3.7}.\\

\noindent{Estimate for $I_1:=\int_S^T \left[E_u^w(t)\right]' \left(u_t(t),u(t)\right)_{L^2(\Omega)}\,dt$}.
\medskip

Considering Cauchy-Schwarz and Poincar\'e inequalities one has:
\begin{eqnarray}\label{eq3.9}
|I_1|&\leq& \sqrt{\lambda_1} \int_S^T\left|\left[E_u^w(t)\right]'\right|||u_t(t)||_2 ||\nabla u(t)||_2\,dt\\
&\leq&\sqrt{\lambda_1}\int_S^T[-E_u^w(t)]'E_u^w(t)\,dt\nonumber\\
&=& \frac{\sqrt{\lambda_1}}{2}\int_S^T- \frac{d}{dt}[E_u^w(t)^2]\,dt \leq -\frac{\sqrt{\lambda_1}}{2}\left([E_u^w(T)]^2 - [E_u^w(S)]^2\right).\nonumber\\
&\leq&  \frac{\sqrt{\lambda_1}}{2}E_u^w(0)E_u^w(S).\nonumber
\end{eqnarray}

\medskip
\noindent{Estimate for $I_2:=-\int_S^T E_u^w(t)||\nabla u_t(t)||_2^2 \left(\nabla u_t(t),\nabla u(t)\right)_{L^2(\Omega)}\,dt$}.
\medskip

Here, we have
\begin{eqnarray}\label{eq3.10}
|I_2|&\leq& \int_S^T E_u^w(t)||\nabla u_t(t)||_2^2 ||\nabla u_t(t)||_2||\nabla u(t)||_2\,dt\\
&\leq&  \int_S^T E_u^w(t)||\nabla u_t(t)||_2^2\left[\frac12 ||\nabla u_t(t)||_2^2 + \frac12||\nabla u(t)||_2^2\right]\,dt\nonumber\\
&\leq& \frac12  \int_S^T E_u^w(t)||\nabla u_t(t)||_2^4\,dt  + \frac12 \int_S^T E_u^w(t)||\nabla u_t(t)||_2^2 ||\nabla u(t)||_2^2\,dt\nonumber\\
&\leq& \frac12  E_u^w(0)\int_S^T ||\nabla u_t(t)||_2^4\,dt  +  [E_u^w(0)]^2\int_S^T ||\nabla u_t(t)||_2^4 \,dt + \frac14\int_S^T [E_u^w(t)]^2\,dt.\nonumber
\end{eqnarray}

\medskip
\noindent{Estimate for $I_3:= - \left[E_u^w(t)\left(u_t,u(t)\right)_{L^2(\Omega)} \right]_S^T$}.
\medskip

Notice that, since
\begin{eqnarray*}
|E_u^w(t)(u_t(t),u(t))_{L^2(\Omega)}|\leq \sqrt{\lambda_1}E_u^w(0)E_u^w(t),
\end{eqnarray*}
we can deduce, in view of (\ref{eq3.2}), that:
\begin{eqnarray}\label{eq3.11}
&&\left|\left[E_u^w(t)\left(u_t,u(t)\right)_{L^2(\Omega)} \right]_S^T\right|\leq  \sqrt{\lambda_1}E_u^w(0)[E_u^w(T)+E_u^w(S)]\\
&&=\sqrt{\lambda_1}E_u^w(0)\left[2E_u^w(S)-  \int_S^T ||\nabla u_t(t)||_2^4 \,dt\right]\leq 2\sqrt{\lambda_1}E_u^w(0)E_u^w(S).\nonumber
\end{eqnarray}

Therefore, considering (\ref{eq3.7})-(\ref{eq3.11}) we obtain the following inequality
\begin{eqnarray*}
\frac34 \int_S^T [E_u^w(t)]^2\, dt&\leq& 2\lambda_1^2\int_S^T||\nabla u_t(t)||_2^4 \,dt + \frac{\sqrt{\lambda_1}}{2}E_u^w(0)E_u^w(S)\\
&+&\left(E_u^w(0)+\frac12[E_u^w(0)]^2\right)\int_S^T ||\nabla u_t(t)||_2^4\,dt + 2{\sqrt{\lambda_1}}E_u^w(0)E_u^w(S),\nonumber
\end{eqnarray*}
from which, taking the energy identity (\ref{eq3.2}) into account, we can deduce that
\begin{eqnarray}\label{eq3.13}
 \int_S^T [E_u^w(t)]^2\, dt \leq C E_u^w(S),
\end{eqnarray}
where $C:=\frac23 \left[4\lambda_1^2+ (5\sqrt{\lambda_1}+1)E_u^w(0)+ 2[E_u^w(0)]^2\right]$.

\medskip

To conclude our analysis, we recall the important result whose proof can be found in \cite{Komornik}.
\begin{lemma}\label{th1.2}
Assume that $E:\mathbb{R}_+ \rightarrow \mathbb{R}_+$ is a non-increasing function and there are two positive constants $\alpha,$ $C,$ such that
\begin{align}\label{C}
\int_t^\infty E^{\alpha+1}(s) \, ds \leq C [E(0)]^\alpha E(t), \quad \forall t\in \mathbb{R}_+.
\end{align}
Then we have
\begin{align*}
E(t) \leq E(0) \left(\frac{C+\alpha t}{C + \alpha C} \right)^{-1/\alpha}.\quad \forall t\geq C.
\end{align*}
\end{lemma} 

\medspace

By applying the above lemma with \( \alpha = 1 \), we can derive what seems to be the optimal algebraic decay rate estimate for mild solutions to problem (\ref{eq1.1}):  $$E_u^w(t) \leq \frac{2C}{C+t}.$$\\

We are now in a position to present the second main result of this article:
\begin{theorem}\label{Theorem2}
The weak energy $E_u^w(t)$ of regular and weak solutions to problem (\ref{eq1.1}) have the following decay rate estimate:
\begin{align*}
E_u^w(t) \leq E_u^w(0)  \frac{2C}{C+t},\quad \forall t\geq C,
\end{align*}
and for some $C>0$.
\end{theorem}

\medskip
\section{Appendix}
\medskip

For the sake of completeness, we will present some results in this Appendix on the theory of Semigroups of Nonlinear Operators, following the approach of \cite{Woo}.\\

Let $X$ be a Banach space with its dual $X'$ with norms $||\cdot||$ and $||\cdot||_*$ respectively. We denote
\begin{equation}\label{eq4.1}
|||X_0|||=\inf\{||x||: x\in X_0\},
\end{equation}
for any nonempty subset $X_0$ of $X$. The duality map $F$ of $X$ is the multi-valued mapping from $X$ to $X'$ defined by
\begin{eqnarray}\label{eq4.2}
Fx:=\{f\in X': \hbox{Re}\left<f,x\right>_{X',X}=||x||^2=||f||_*^2 \}.
\end{eqnarray}

In the present discussion, given $X_0$ a subset of $X$, we will denote by $\hbox{Cont}(X_0)$  the set of all contraction operators on $X_0$, that is, the set of all operators $T:X_0\rightarrow X_0$ such that
$$||Tx - Ty|| \leq ||x-y||$$ for all $x,y \in X_0$.\\

We have the following preliminary results.
\begin{theorem}\label{F_uniformly_continuous}
	If $X'$ is uniformly convex, then $F$ is single-valued and it is uniformly continuous on any bounded set of $X$. In other words, for each $\varepsilon>0$ and $M>0$, there is $\delta>0$ such that $||x||<M$ and $||x-y||<\delta$ implies $||Fx-Fy||<\varepsilon$.
\end{theorem}
	
	

\begin{theorem}\label{derivada}
	Let $x(t)$ be an $X$-valued function on an interval of real numbers. Suppose $x(t)$ has a weak derivative $x'(t)\in X$ at $x=s$. If $||x(t)||$ is also differentiable at $t=s$, then
	$$||x(s)||\frac{d}{ds}||x(s)||=\hbox{Re}\left<x'(s),f\right>,$$
	for every $f\in Fx(s)$.
\end{theorem}
	

Next, let $A$ be an operator, not necessarily linear, in $X$. $A$ is said to be \underline{dissipative} if for every $x,y\in D(A)$, $x'\in Ax$ and $y'\in Ay$, there exists $f\in F(x-y)$ such that $$\hbox{Re}\left<f, x'-y'\right>\leq 0$$
and \underline{maximal dissipative} ($m$-dissipative) if $A$ is dissipative and $$R(I-\alpha A)=X,$$ for all $\alpha>0$.\\

Now, given $A$ a dissipative operator (not necessarily linear), we can define $$J_\alpha:= (I-\alpha A)^{-1},$$ for all $\alpha>0$, with $D(J_\alpha)=R(I-\alpha A)$ and $R(J_\alpha)=D(A)$, and set $$A_\alpha:=\alpha^{-1} (J_\alpha - I)$$ the so called \underline{Yosida operator}. Notice that both $J_\alpha$ and $A_\alpha$ are single-valued with $D_\alpha=D(J_\alpha)=D(A_\alpha)$, $J_\alpha x=y$ for $x\in D_\alpha$ if and only if $x\in (I-\alpha A)y=y-\alpha Ay$ for $y\in D(A)$, by definition. Also, $J_\alpha$ is a contraction and $A_\alpha$ is Lipschitz continuous with $$||A_\alpha x - A_\alpha y||\leq 2\alpha^{-1}||x-y||,$$
for $x,y\in D_{\alpha}$.\\

We have the useful result:
\begin{theorem}\label{TeoApp1} Let $A$ be dissipative and let $\alpha>0$. Then:
	
(i)~$A_\alpha$ is dissipative.

(ii)~$A_\alpha x \in A J_\alpha x$ and $|||AJ_\alpha x||| \leq ||A_\alpha x||$, for $x\in D_\alpha$.

(iii) If $x\in D(A)\cap D_\alpha$, then $$|||A J_\alpha x||| \leq ||A_\alpha x|| \leq |||Ax|||.$$

(iv)~$\lim\limits_{\alpha \rightarrow 0} J_\alpha x=x$ for $x\in D(A) \cap D_\alpha$ and $\alpha>0$.

\end{theorem}




\medskip
\subsection{Generation of Semigroups}
\medskip
The main goal of this section is to show the existence of a semigroup associated with a nonlinear $m$-dissipative operator $A$ on the Banach space $X$. We will start by introducing the concept of nonlinear semigroups.\\

Let $X_0$ be a subset of $X$. Let $\{T(t): t\geq0\}$ be a family of one-parameter operators, not necessarily linear, from $X_0$ into itself satisfying the following conditions:

(i) $T(0)=I$, ~$T(s)T(t) = T(t+s)$, for $s,t\geq 0$,

(ii) For $x\in X_0$, $T(t)x$ is strongly continuous in $t\geq 0$,

(iii) $T(t)\in \hbox{Cont} (X_0)$ for $t\geq 0$.

Then we call such family $\{T(t): t\geq 0$\} a \underline{nonlinear contraction semigroup} on $X_0$ and we define the infinitesimal generator $A_0$ of the semigroup $\{T(t): t\geq 0\}$ by $$A_0x = \lim_{h\rightarrow 0_+} \frac{T(h)-I}{h}x$$ and the weak generator $A'$ by $$A'x = \underset{h\rightarrow 0_+}{\text{w-lim}} \frac{T(h)-I}{h}x ,$$
if the right side exists in $X$, i.e., for all $\phi\in X'$,
$$\left<\phi,A'x\right>=\left<\phi,\frac{T(h)x-x}{h}\right>.$$

Before to state the main result, we need of the following auxiliary results:


\begin{theorem}\label{TeoApp3}
Let $A$ be $m$-dissipative and let $\alpha>0$. Then problem:
\begin{equation}\label{3.1}
\left\{
\begin{aligned}
& \frac{d}{dt}u(t,x)= A_\alpha u(t,x),~t\geq0\\
&u(0,x)=x\in X,
\end{aligned}
\right.
\end{equation}
has a unique solution $u(t,x)\in C^1([0,+\infty);X)$ and
$$||u(t,x)-u(t,y)|| \leq ||x-y||,~\forall x,y \in X.$$
\end{theorem}
\begin{proof}
	Since $A_\alpha x$ is Lipschitz continuous, uniformly in $X$, the problem (\ref{3.1}) has a unique solution $u(t,x)\in C^1([0,+\infty);X)$. For $x,y\in X$, put $w(t)=u(t,x)-u(t,y)$, for $t\geq 0$. Then $w(t)\in C^1([0,+\infty);X)$ and
	\begin{equation*}
		\left\{
		\begin{aligned}
			& \frac{d}{dt}w(t)= A_\alpha u(t,x)-A_\alpha u(t,y),~t\geq0\\
			&w(0)=x-y.
		\end{aligned}
		\right.
	\end{equation*}

	Now, as $||w(t)||$ is absolutely continuous, $||w(t)||$ is differentiable for a.e. $t>0$ and, by  Theorem \ref{derivada} we get for a.e. $s>0$ that
	$$||w(s)||\frac{d}{ds}||w(s)||=\hbox{Re}\left<w'(s),f\right>=\hbox{Re}\left<A_\alpha u(s,x)-A_\alpha u(s,y),f\right>,$$
	for every $f\in Fw(s)$. On the other hand, since $A_\alpha$ is dissipative, by Theorem \ref{TeoApp1} there exists $f_s\in Fw(s)$ such that
	$$\hbox{Re}\left<A_\alpha u(s,x)-A_\alpha u(s,y),f_s\right>\leq 0.$$
	Thus
	$$||w(s)||\frac{d}{ds}||w(s)||\leq 0,\hbox{ for a.e.} s>0.$$
	Therefore, as 
	$$||w(t)||^2-||w(0)||^2=\int_0^t\frac{d}{ds}||w(s)||^2\,ds\leq 0, \hbox{ for }t\geq 0,$$
	we obtain
	$$||u(t,x)-u(t,y)|| \leq ||x-y||,\hbox{ for } t\geq 0.$$	
\end{proof}

\begin{theorem}\label{TeoApp4}
Let $A$ be $m$-dissipative and $\alpha>0$. Then, there exists a semigroup $\{T_\alpha(t):t\geq 0\}$ on $X$ with its infinitesimal generator $A_\alpha$ such that for $x\in X$, $T_\alpha(t)x\in C^1([0,+\infty);X)$ and  $$\frac{d}{dt} T_\alpha(t) x = A_\alpha T_\alpha(t)x.$$
\end{theorem}
\begin{proof}
	Let $u(t,x)$ be the unique solution of the problem (\ref{3.1}), and put $T_\alpha(t)x=u(t,x)$, for $x\in X$ and $t\geq 0$. Then $T_\alpha(t)x$ is strongly continuous in $t\geq 0$. Also, since for all $x,y\in X$
	$$||T_\alpha(t)x-T_\alpha(t)y||=||u(t,x)-u(t,y)|| \leq ||x-y||,$$
	we have that $T_\alpha(t)\in\hbox{Cont}(X)$ for all $t\geq 0$. Finally, the semigroup property follows from the unicity of the solution of (\ref{3.1}). Indeed, let $x\in X$ and set $E_x=\{s\geq0;\, u(s,x)\in X\}$. Then from the uniqueness of solution of (\ref{3.1}) we see that $u((t+s),x)=u(t,u(s,x))$ for $t\geq0$ and $s\in E_x$, that is, $T_\alpha(t+s)x=T_\alpha(t)T_\alpha(s)x$, for $t\geq0$ and $s\in E_x$. Thus we have
	$$T_\alpha(t+s)x=T_\alpha(t)T_\alpha(s)x,\hbox{ for } t,s\geq0.$$
	Moreover, from (\ref{3.1})
	$$T_\alpha(0)x=u(0,x)=x,\hbox{ for }x\in X,$$
	hence $T_\alpha(0)=I$. Therefore $\{T_\alpha(t);\, t\geq 0\}$ is a semigroup on $X$.
\end{proof}

The main theorem of this section reads as follows:
\begin{theorem}\label{main theo}
Let $X'$ uniformly convex and let $A$ be $m$-dissipative. Then, there exists a semigroup $\{T(t):t\geq0\}$ on $\overline{D(A)}$.
\end{theorem}
\begin{proof}

For $x\in D(A)$ put
$$w_{\alpha,\beta}(t):= T_\alpha(t)x - T_\beta(t)x,~\hbox{ for } \alpha, \beta>0.$$

We shall show that
$$\lim_{\alpha,\beta\rightarrow 0_+}w_{\alpha,\beta}(t)=0~\hbox{for}~t\geq0.$$

Indeed, we first notice that, by theorem \ref{TeoApp4}, we have 
\begin{eqnarray}\label{3.3}
	||A_\alpha T_\alpha(t)x|| \leq ||A_\alpha x|| \leq |||Ax|||,~\hbox{ for } x \in D(A)\hbox{ and } t\geq 0.
\end{eqnarray}Then, by (\ref{3.3}) we have:
\begin{eqnarray*}
||w_{\alpha,\beta}(s)||&\leq& \int_0^t ||A_\alpha T_\alpha(t) x - A_\beta T_\beta(t)x||\,dt\\
&\leq& 2|||Ax|||s,~\hbox{ for }s\geq0.
\end{eqnarray*}

Now, setting $$v_{\alpha,\beta}(s):= J_\alpha T_\alpha(s) - J_\beta T_\beta(s),$$
we obtain
\begin{align*}
	||w_{\alpha,\beta}(s) - v_{\alpha,\beta}(s)|| & \leq ||(I-J\alpha)T_\alpha(s)x||+||(I-J_\beta)T_\beta(s)x||\\
	&\leq\alpha||A_\alpha T_\alpha(s)x||+\beta||A_\beta T_\beta(s)x||\\
	&\leq (\alpha+\beta)|||Ax|||.
\end{align*}
Hence
$$||v_{\alpha,\beta}(s)||\leq(2s+\alpha+\beta)|||Ax|||.$$

The dissipative nature of $A$, along with Theorem \ref{TeoApp1}, implies that
\begin{eqnarray}\label{3.4}
&&\hbox{Re}\left(A_\alpha T_\alpha(s)x - A_\beta T_\beta(s) x, F v_{\alpha\,\beta}(s) \right)\leq0,
\end{eqnarray}
and then
\begin{eqnarray*}
\hbox{Re}\left(A_\alpha T_\alpha(s)x - A_\beta T_\beta(s) x, F w_{\alpha\,\beta}(s) \right)\leq 2|||Ax|||\, ||Fw_{\alpha,\beta}(s)-Fv_{\alpha, \beta}(s)||.
\end{eqnarray*}
Therefore, by using theorem \ref{derivada} we conclude that,
\begin{eqnarray*}
||w_{\alpha,\beta}(t)||^2&=&\int_0^t \frac{d}{ds}||w_{\alpha,\beta}(s)||^2\,ds\\
&=& 2\int_0^t \hbox{Re}\left<A_\alpha T_\alpha(s)x - A_\beta T_\beta(s)x, Fw_{\alpha,\beta}(s) \right>\,ds\\
&\leq& 4|||Ax|||\,\int_0^t ||Fw_{\alpha,\beta}(s)-Fv_{\alpha,\beta}(s)||\,ds.
\end{eqnarray*}

Next we observe that,  for $t_0>0$ the set
$$\{w_{\alpha,\beta}(s),v_{\alpha,\beta}(s); 0\leq s\leq t_0,~0<\alpha, \beta \leq 1\}$$
is bounded. Additionally, $w_{\alpha,\beta}(s)-v_{\alpha,\beta}(s)$ converges uniformly in $s\in[0,t_0]$ to zero as $\alpha,\beta \rightarrow 0_+$. Given that $X'$ is uniformly convex, from Theorem \ref{F_uniformly_continuous} we conclude that $F$ is uniformly continuous on any bounded set of $X$. Consequently:
$$\lim_{\alpha,\beta \rightarrow 0^+}w_{\alpha,\beta}(t)=0~\hbox{ uniformly in }t\in [0,t_0],$$
that is, for $x\in D(A)$:
\begin{eqnarray}\label{3.5}
\lim_{\alpha,\beta \rightarrow 0^+} ||T_\alpha(t)x-T_\beta(t)x||=0~\hbox{ uniformly in }t\in [0,t_0],
\end{eqnarray}
and, since  $T_\alpha \in \hbox{Cont} (X)$, (\ref{3.5}) holds true for $x\in \overline{D(A)}$. 

In view of the previous arguments, if we define $T(t)$ by
\begin{eqnarray}\label{3.6}
T(t)x:= \lim_{\alpha \rightarrow 0} T_\alpha(t)x,~\hbox{ for }~x\in \overline{D(A)},
\end{eqnarray}
we conclude that $\{T(t): t\geq 0\}$ is semigroup on $\overline{D(A)}$. Indeed, it is sufficient to show that $T(t)$ maps $\overline{D(A)}$ into itself. Let $x\in D(A)$. As we have that $||J_\alpha T_\alpha(t)x-T_\alpha(t)x||\leq\alpha|||Ax|||$, then
$$\lim\limits_{\alpha \rightarrow 0^+}J_\alpha T_\alpha(t)x=T(t)x$$
uniformly in $t\in[0,t_0]$, and $T(t)x\in D(A)$ since $J_\alpha T_\alpha(t)x\in D(A)$. 

Moreover, as $T(t)\in Cont(D(A))$, it follows that $T(t)$ maps $ \overline{D(A)}$ into itself. This concludes the proof.
\end{proof}

  \subsection*{Statements \& Declarations}
Research of Marcelo M. Cavalcanti is partially supported by the CNPq Grant 300631/2003-0.
Research of Val\'eria N. Domingos Cavalcanti is partially supported by the CNPq Grant 304895/2003-2.  
Research of Cintya Akemi Okawa is partially supported by the Coordena\c c\~ao de Aperfei\c coamento de Pessoal de N\'ivel Superior -Brasil (CAPES) - Finance Code 001.

\subsection*{Funding and/or Conflicts of interests/Competing interests}
No funding has been received other than what was mentioned, and there are no conflicts of interest or competing interests to declare.

\end{document}